\documentclass{article}
\usepackage{mathrsfs}
\usepackage{bbm}
\usepackage{amsthm}
\usepackage{amsmath}
\usepackage{amsfonts}
\usepackage{cases}
\usepackage{array}
\usepackage{cite}
\usepackage{appendix}
\usepackage{float}
\usepackage{graphicx}
\usepackage{extarrows}

\newtheorem{definition}{Definition}[section]
\newtheorem{theorem}{Theorem}[section]
\newtheorem{proposition}{Proposition}[section]

\begin{document}
\baselineskip 6mm
\renewcommand{\thesection}{\arabic{section}}
\renewcommand{\theequation}{\thesection.\arabic{equation}}
\numberwithin{equation}{section}
\title{Mechanical Proving the Symplecticity of Partitioned Runge--Kutta Methods for Determinate and Stochastic Hamiltonian Systems\footnote{Supported by Youth Foundation of Shandong Natural Science Foundation. (No. ZR2024QA201)}}
\author{Xiaojing Zhang
\footnote{202314122@sdtbu.edu.cn, School of Mathematics and Information Science, Shandong Technology and Business University, 264005, Yantai, Shandong, PR China}}

\date{}
\maketitle
\begin{abstract} We propose a new method to prove the partitioned Runge--Kutta methods with symplectic conditions  for determinate  and stochastic Hamiltonian systems are symplectic. We utilize Gr\"obner basis technology which is the one of symbolic computation method based on computer algebra theory and geometrical mechanical proving theory. In this approach, from determinate Hamilton's equations, we get the relations of partial differentials which are regarded as polynomials of plenty variables marked indeterminates. Then, we compute the Gr\"obner basis of above polynomials, and the normal form of symplectic expression, which is as the middle expression, with respect to the Gr\"obner basis. Then, we compute the Gr\"obner basis of symplectic conditions and the normal form of the middle expression with respect to above Gr\"obner basis, and get that the normal form is zero, which complete the proof. We also develop this procedure to the stochastic Hamiltonian systems case and get similar result. In this paper, the new try provide us a new idea to prove the structure-preservation laws of another numerical methods, including the energy conservation law, the momentum conservation law and so on.

{\flushleft $Keywords:$ }  partitioned symplectic Runge-Kutta methods; symbolic computation; mechanical theorem-proving; Hamiltonian systems.
\end{abstract}

\section{Introduction} \setcounter{equation}{0}\hskip\parindent
Hamilton's equations were first found by William Rowan Hamilton in 1834, and appears frequently among the thousands of equations. Many physics equations can be change to the form of  Hamiltonian systems, for example, wave equation, schr\"odinger equation, maxwell equation, which describe physic phenomenon such as the trajectory of electrons within an atom,  the disordered and rapid thermal motion of gas molecules. The Hamiltonian provide a unified mathematical tool to describe those seemingly distinct physic systems. As a mathematical description of system energy, the Hamiltonian not only offers a concise expression of dynamic equations but also reveals the deep symmetries and conservation laws of complex physical systems. One of the conservation laws is symplectic structure conservation for Hamiltonian systems, because the flow of Hamiltonian systems is a symplectic transformation\cite{geometric}. The symplectic structure-preserving numerical methods for Hamiltonian systems play an important role in inheriting intrinsic property of Hamiltonian systems. Due to the superiority of symplectic methods in the long-time numerical simulation, lots of scholars have begun to focus on the construction of symplectic methods for Hamiltonian systems\cite{feng,ruth}. The symplectic Runge--Kutta methods are frequently applied to solve the numerical solution of the Hamiltonian systems, which were first studied by Sanz-Serna, Suris and Lasagni\cite{sanz,suris,lasagni}. 
 
An important class of symplectic partitioned Runge--Kutta(PRK) methods is the Lobatto IIIA-IIIB pair, first discussed by Sun \cite{sungeng}.  We have the conclusion that the PRK methods for determinate Hamiltonian ODEs are symplectic iff the coefficients satisfy the symplectic conditions.
At the same time, multi-symplectic PRK methods for Hamiltonian PDEs are studied by Hong for
the first time\cite{hong}.
Recently, the stochastic Hamiltonian systems has been widely concerned, since random effects are needed to take
into account when stochasticity occurs from internal and external disturbances on the differential equations. The stochastic Hamiltonian systems also have an intrinsic property which is the preservation of the symplectic structure\cite{milstein,milstein2}. Many researchers develop the symplectic methods and its convergence analysis for stochastic Hamiltonian systems\cite{WLJ,hong2,hong3}. For stochastic Hamiltonian systems with multiplicative noise, the symplectic conditions for Runge--Kutta methods was first discovered by Ma et al.\cite{maq,maq2}. 

Due to the complexity of the PRK, the proof of symplecticity involves plenty of equations and complicated substituting. If there has a  way to accomplish the proof by computer, it will avoid the complexity. In this paper, we propose new proofs of the symplectic conditions for partitioned Runge--Kutta method for determinate and stochastic Hamiltonian systems, where we utilize Gr\"obner basis technology in symbolic computation theory. This new approach is established based on geometrical mechanical proving theory , which change the process of proving by words to proving by computer. Thus, we simplify the proof in \cite{maq,maq2}. As for the background of mechanical proving, the mechanical proving is based on algorithmic algebra theory refer to \cite{Mishra,cox,chen1}. In 2006, Bruno Bucheberger first proposed a new algorithm for finding the basis elements of the residue class ring of a zero-dimensional polynomial ideal, which is called the Gr\"obner basis algorithm later\cite{Buch}. The mathematician Wenjun Wu first proposed the character set theory which expanded the theory of mechanical proving\cite{wu,wu2,wu3}. In this paper, the new try provide us a new idea to prove the structure-preservation laws of another numerical methods, including the energy conservation law, the momentum conservation law and so on.

The present paper is organized as follows. In the second section, we introduce some definitions and conclusions about symplectic conditions of PRK methods for determinate and stochastic Hamiltonian systems. We also give some definitions and propositions about the Gr\"obner basis and its property in section 3.  The most important part in this research is the mechanical proving the symplecticity of PRK methods, which are given in the fourth section. The fifth section is the conclusions and prospects of this article.

\section{Symplectic Conditions for PRK Method}

\subsection{PRK Method for Determinate Hamiltonian Systems}

 \qquad In this section, we list some theorems and definitions to introduce some primary notations about symplectic conditions for PRK method for determinate Hamiltonian systems.

  We consider determinate Hamiltonian systems as follows,
\begin{equation}\label{dHs}
\begin{aligned}
&\dot{p}(t)=-H_q(p(t),q(t)),\quad t\in [ 0,T], \\
&\dot{q}(t)=H_p(p(t),q(t)),\quad t\in [ 0,T],
\end{aligned}
\end{equation}
where $p,q$ are  $\mathbb{R}$-valued, and $H: \mathbb{R}^2\rightarrow \mathbb{R}$ is a sufficiently smooth function. The equivalent form is
\begin{equation*}
y=\left(\begin{array}{c}
p\\
q
\end{array}\right)
,\qquad \dot{y}=J^{-1} \nabla H(y),
\end{equation*}
where 
\begin{equation*}
J=\left[\begin{array}{cc}
0 & 1\\
-1 & 0
\end{array}\right]
\end{equation*}

The symplectic property of Hamiltonian systems is described as follows:
\begin{definition}\cite[Definition VI 2.2]{geometric}
A differentiable map $g: U\rightarrow\mathbb{R}^2$(where $U\subset \mathbb{R}^2$ is an open set) is called {\bf symplectic} if the Jacobian matrix $g'(p,q)$ is everywhere symplectic, i.e., if
\begin{equation*}
g'(p,q)^\top J g'(p,q)=J.
\end{equation*}
\end{definition}

\begin{theorem}\cite[Theorem VI 2.4]{geometric}
Let $H(p,q)$ be a twice continuously differentiable function on $U\subset\mathbb{R}^2$. Then, for each fixed $t$, the flow $\phi_t$ is a symplectic transformation wherever it is defined.
\end{theorem}
Let $h=T/N$ and $t_n=nh$ be the grid node in $[ 0,T]$, and $p^n\approx p(t_n)$ be the numerical approximation of $p(t_n)$. The partitioned Runge--Kutta methods are as follows.
\begin{definition}\cite[Definition II 2.1]{geometric}
Let $b_i,\ a_{ij}$ and $\hat{b}_i,\ \hat{a}_{ij}$ be the coefficients of two $s$-staged Runge--Kutta methods, An {\bf $s$-staged partitioned Runge--Kutta method} for  \eqref{dHs} is given by
\begin{equation}\label{bpRK}
\begin{aligned}
&p^{n,i}=-H_q(p^{n-1}+h\sum_{j=1}^{s}a_{ij}p^{n,j},\ q^{n-1}+h\sum_{j=1}^{s}\hat{a}_{ij}q^{n,j}),\quad i=1,2,\dots,s,\\
&q^{n,i}=H_p(p^{n-1}+h\sum_{j=1}^{s}a_{ij}p^{n,j},\ q^{n-1}+h\sum_{j=1}^{s}\hat{a}_{ij}q^{n,j}),\quad i=1,2,\dots,s,\\
&p^{n}=p^{n-1}+h\sum_{i=1}^{s}b_i p^{n,i},\\
&q^{n}=q^{n-1}+h\sum_{i=1}^{s}\hat{b}_i q^{n,i},\\
\end{aligned}
\end{equation}
where $p^{n,i},q^{n,i}$ are the middle value between $p^{n-1}$ and $p^n$, $q^{n-1}$ and $q^n$ respectively.
\end{definition}
A numerical method is called symplectic if it can preserve the symplectic structure of Hamiltonian system.
\begin{definition}\cite[Definition VI 3.1]{geometric}
A numerical one-step method is called {\bf symplectic} if the one-step map
\begin{equation*}
y_{n+1}=\phi_h(y_n)
\end{equation*}
is symplectic whenever the method is applied to a smooth Hamiltonian system.
\end{definition}
The symplectic conditions for partitioned Runge--Kutta methods are some relations of coefficients, which are given in the following theorem.
\begin{theorem}\cite[Theorem VI 4.6]{geometric}
If the coefficients of PRK methods \eqref{bpRK} satisfy
\begin{align}
&b_i-\hat{b}_i=0,\quad i=1,2,\dots,s,\label{sc1}\\
& b_i\hat{a}_{ij}+\hat{b}_ja_{ji}-b_i\hat{b}_j=0,\quad i,j=1,2,\dots,s,\label{sc2}
\end{align}
then it is symplectic.
\end{theorem}

\subsection{Partitioned Runge--Kutta Method for Stochastic Hamiltonian Systems}\par
\qquad We consider stochastic Hamiltonian systems as follows,
\begin{equation}\label{sHs}
\begin{aligned}
&dp=-H_q(p,q)dt-\tilde{H}_q(p,q)\circ dB(t),\ \  p(t_0)=p_0,\\
&dq=H_p(p,q)dt+\tilde{H}_p(p,q)\circ dB(t),\ \ q(t_0)=q_0,
\end{aligned}
\end{equation}
where $H,\tilde{H}$ are $\mathbb{R}$-valued sufficiently smooth functions,  $B(t)$ is a standard one-dimensional Brownian motion, and the small circle '$\circ$' before $dB(t)$ denotes stochastic differential equations with Stratonovich type. 
\begin{theorem}\cite[theorem 3.3]{WLJ}
Stochastic Hamiltonian systems preserve symplectic structure, that is 
\begin{equation}
dp(t)\wedge dq(t)=dp_0\wedge dq_0, \quad \forall t\geq 0, a.s.\label{wedge}
\end{equation}
At the same time, equation \eqref{wedge} is satisfied if and only if 
\begin{equation}
B(t)^\top JB(t)=J, \quad \forall t\geq 0, a.s.
\end{equation}
where
\begin{equation*}
B(t):=\frac{\partial (p(t),q(t))}{\partial (p_0,q_0)}.
\end{equation*}
\end{theorem}
A numerical method is called symplectic if it can preserve the symplectic structure of stochastic Hamiltonian systems.
\begin{definition}\cite{milstein2}
A numerical one-step method $ (p^{n-1},q^{n-1})\mapsto (p^{n},q^{n})$ applied to \eqref{sHs} is called {\bf symplectic} if we have
$$dp^{n}\wedge dq^{n}=dp^{n-1}\wedge dq^{n-1},\quad\forall n\geq 0.$$
or
\begin{equation}
\frac{\partial(p^n,q^n)}{\partial(p^{n-1},q^{n-1})}^\top J\frac{\partial(p^n,q^n)}{\partial(p^{n-1},q^{n-1})}\\
=J,\quad\forall n\geq 0.
\end{equation}
\end{definition}
\begin{definition}
The {\bf stochastic $s$-staged PRK method} for \eqref{sHs} are as follows.
\begin{equation}\label{spRK}
\begin{aligned}
&p^{n,i}=p^{n-1}-h\sum_{j=1}^{s}a_{ij}H_q(p^{n,j},q^{n,j})- \Delta B_n\sum_{j=1}^{s}\alpha_{ij}\tilde{H}_q(p^{n,j},q^{n,j})\quad i=1,2,\dots,s,\\
&q^{n,i}=q^{n-1}+h\sum_{j=1}^{s}\hat{a}_{ij}H_p(p^{n,j},q^{n,j})+ \Delta B_n\sum_{j=1}^{s}\hat{\alpha}_{ij}\tilde{H}_p(p^{n,j},q^{n,j}) \quad i=1,2,\dots,s,\\
&p^{n}=p^{n-1}-h\sum_{i=1}^{s}b_i H_q(p^{n,i},q^{n,i})- \Delta B_n\sum_{i=1}^{s}\beta_i\tilde{H}_q(p^{n,i},q^{n,i}),\\
&q^{n}=q^{n-1}+h\sum_{i=1}^{s}\hat{b}_i H_p(p^{n,i},q^{n,i})+ \Delta B_n\sum_{i=1}^{s}\hat{\beta}_i\tilde{H}_p(p^{n,i},q^{n,i}),\\
\end{aligned}
\end{equation}
where $\Delta B_n=B(t_n)-B(t_{n-1})$. 
\end{definition}
The symplectic conditions of stochastic PRK methods are as follows.
\begin{theorem}\cite[theorem 3.1]{maq2}
The numerical scheme \eqref{spRK} preserve symplectic structure, i.e., $dp^{n}\wedge dq^{n}=dp^{n-1}\wedge dq^{n-1},$ if the coefficients of \eqref{spRK} satisfy the relations
\begin{equation}\label{spRKsc}
\begin{aligned}
&b_i\hat{b}_j-b_i\hat{a}_{ij}-\hat{b}_j a_{ji}=0,\quad i,j=1,2,\dots,s,\\
&\beta_i\hat{b}_j-\beta_i\hat{a}_{ij}-\hat{b}_j\alpha_{ji}=0,\quad i,j=1,2,\dots,s,\\
&b_i\hat{\beta}_j-b_i\hat{\alpha}_{ij}-\hat{\beta}_j a_{ji}=0,\quad i,j=1,2,\dots,s,\\
&\beta_i\hat{\beta}_j-\beta_i\hat{\alpha}_{ij}-\hat{\beta}_j\alpha_{ji}=0,\quad i,j=1,2,\dots,s,\\
&b_i=\hat{b}_i, \quad \beta_i=\hat{\beta}_i,\quad i=1,2,\dots,s.
\end{aligned}
\end{equation}
\end{theorem}

\section{The Gr\"obner Basis}

\qquad The Grobner basis technology effectively calculates a well-functioning set of generators according to an any set of generators of polynomial ideal. The above well-functioning set of generators is called the Gr\"obner basis. The Gr\"obner basis can be used to judge whether any polynomial belongs to the ideal and can deals with many complex calculation problems. It has now become a fundamental tool in the theory of computer algebra, commutative algebra and algebraic geometry mechanical proof.
\begin{definition}\cite[(3.1)Definition]{cox}\label{defi1}
Fix a monomial order $>$ on polynomial ring of indeterminates $x_1,...,x_n$ with coefficients in number field $K$: $K\left[x_1,...,x_n\right]$, and let $I\subset K\left[x_1,...,x_n\right]$ be an ideal. A {\bf Gr\"obner basis} for $I$ is a finite collection of polynomials $G=\{g_1,...,g_n\}\subset I$ with the property that for every nonzero $f\in I$, the leading term of $f$ is reduced by the leading term of $g_i$ for some $i$.
\end{definition}

Let monomial order $>$ and let $I\subset K\left[x_1,...,x_n\right]$ be an ideal. Due to the uniqueness of remainders with respect to(w.r.t.) Gr\"obner basis, reduction(division) of $f\in K\left[x_1,...,x_n\right]$ by a Gr\"obner basis for $I$ produces an expression $f=g+r$ where $g\in I$ and no term in $r$ is reducible(divisible) by any leading term of element in $I$. If $f=g'+r'$ is any other such expression, then $r=r'$.

\begin{proposition}\cite{cox,chen1}\label{pro1}
If $G$ is a Gr\"obner basis for $I$, then for any $f\in I$, the remainder on reduction(division) of $f$ w.r.t. $G$ is zero. The remainder is called {\bf the normal form} of $f$ w.r.t. $G$.
\end{proposition}

\section{Mechanical Proving the  Symplecticity of PRK Methods }\hskip\parindent
This section is devoted to proving the symplectic conditions of PRK methods by using of the Gr\"obner basis technology\cite{Mishra,Buch,chen1,wu,wu2,wu3}.

Proving the symplecticity of PRK methods \eqref{bpRK} and \eqref{spRK} are both sufficient to prove 
\begin{align*}
&\frac{\partial(p^n,q^n)}{\partial(p^{n-1},q^{n-1})}^\top J\frac{\partial(p^n,q^n)}{\partial(p^{n-1},q^{n-1})}\\
=&\left[
\begin{array}{cc}\displaystyle
-\frac{d q^n}{d p^{n-1}}^\top\frac{d p^n}{d p^{n-1}}+\frac{d p^n}{d p^{n-1}}^\top\frac{d q^n}{d p^{n-1}}   & \displaystyle-\frac{d q^n}{d p^{n-1}}^\top\frac{d p^n}{d q^{n-1}}+\frac{d p^n}{d p^{n-1}}^\top\frac{d q^n}{d q^{n-1}}\\
\displaystyle-\frac{d q^n}{d q^{n-1}}^\top\frac{d p^n}{d p^{n-1}}+\frac{d p^n}{d q^{n-1}}^\top\frac{d q^n}{d p^{n-1}}   & \displaystyle-\frac{d q^n}{d q^{n-1}}^\top\frac{d p^n}{d q^{n-1}}+\frac{d p^n}{d q^{n-1}}^\top\frac{d q^n}{d q^{n-1}}
\end{array}
\right]\qquad\forall n\geq 1.\\
=&J,
\end{align*}
 It is clear that 
\begin{equation*}
-\frac{d q^n}{d p^{n-1}}\frac{d p^n}{d p^{n-1}}+\frac{d p^n}{d p^{n-1}}\frac{d q^n}{d p^{n-1}} =0,\qquad
-\frac{d q^n}{d q^{n-1}}\frac{d p^n}{d q^{n-1}}+\frac{d p^n}{d q^{n-1}}\frac{d q^n}{d q^{n-1}}=0. 
\end{equation*}
We only need to prove 
\begin{equation}\label{32}
-\frac{d q^n}{d q^{n-1}}\frac{d p^n}{d p^{n-1}}+\frac{d p^n}{d q^{n-1}}\frac{d q^n}{d p^{n-1}} +1=0.
\end{equation}
\begin{theorem}The s-stage PRK methods \eqref{bpRK} with symplectic conditions \eqref{sc1}-\eqref{sc2} for determinate Hamiltonian system \eqref{dHs} are symplectic.
\end{theorem}
\begin{proof}
By denoting
\begin{equation*}
P^{n,i}:=p^{n-1}+h\sum_{j=1}^{s}a_{ij}p^{n,j},\quad Q^{n,i}:=q^{n-1}+h\sum_{j=1}^{s}\hat{a}_{ij}q^{n,j},
\end{equation*}
and from \eqref{bpRK}, we obtain
\begin{subequations}\label{31}
\begin{align}
&\frac{d p^n}{d p^{n-1}}=1+h\sum_{i=1}^s b_i \frac{d p^{n,i}}{d p^{n-1}},\quad \frac{d p^n}{d q^{n-1}}=h\sum_{i=1}^s b_i \frac{d p^{n,i}}{d q^{n-1}},\\
&\frac{d q^n}{d p^{n-1}}=h\sum_{i=1}^s \hat b_i \frac{d q^{n,i}}{d p^{n-1}},\quad \frac{d q^n}{d q^{n-1}}=1+h\sum_{i=1}^s \hat b_i \frac{d q^{n,i}}{d q^{n-1}},\\
&\frac{d p^{n,i}}{d p^{n-1}}=-H_{qp}^i\frac{d P^{n,i}}{d p^{n-1}}-H_{qq}^i\frac{d Q^{n,i}}{d p^{n-1}},\quad \frac{d p^{n,i}}{d q^{n-1}}=-H_{qp}^i\frac{d P^{n,i}}{d q^{n-1}}-H_{qq}^i\frac{d Q^{n,i}}{d q^{n-1}},\quad i=1,2,\dots,s,\\
&\frac{d q^{n,i}}{d p^{n-1}}=H_{pp}^i\frac{d P^{n,i}}{d p^{n-1}}+H_{pq}^i\frac{d Q^{n,i}}{d p^{n-1}},\quad \frac{d q^{n,i}}{d q^{n-1}}=H_{pp}^i\frac{d P^{n,i}}{d q^{n-1}}+H_{pq}^i\frac{d Q^{n,i}}{d q^{n-1}},\quad i=1,2,\dots,s,\\
&\frac{d P^{n,i}}{d p^{n-1}}=1+h\sum_{j=1}^s a_{ij}\frac{d p^{n,i}}{d p^{n-1}},\quad \frac{d P^{n,i}}{d q^{n-1}}=h\sum_{j=1}^s a_{ij}\frac{d p^{n,i}}{d q^{n-1}},\quad i=1,2,\dots,s,\\
&\frac{d Q^{n,i}}{d p^{n-1}}=h\sum_{j=1}^s\hat a_{ij}\frac{d q^{n,i}}{d p^{n-1}},\quad \frac{d Q^{n,i}}{d q^{n-1}}=1+h\sum_{j=1}^s \hat a_{ij}\frac{d q^{n,i}}{d q^{n-1}},\quad i=1,2,\dots,s.
\end{align}
\end{subequations}
where $H_{qp}(P^{n,i},Q^{n,i})$ is written as $H_{qp}^i$ and another second partial derivatives of $H$ are written in similar way. Equations \eqref{31} can be regarded as polynomials with indeterminates $\displaystyle\frac{d p^n}{d p^{n-1}},\frac{d p^n}{d q^{n-1}},\frac{d q^n}{d p^{n-1}},\frac{d q^n}{d q^{n-1}},\frac{d p^{n,i}}{d p^{n-1}}$, $\displaystyle\frac{d p^{n,i}}{d q^{n-1}},\frac{d q^{n,i}}{d p^{n-1}},\frac{d q^{n,i}}{d q^{n-1}},\frac{d P^{n,i}}{d p^{n-1}},\frac{d P^{n,i}}{d q^{n-1}},\frac{d Q^{n,i}}{d p^{n-1}},\frac{d Q^{n,i}}{d q^{n-1}}$. We take $s=2$ as an example. Let the indeterminates order be $\displaystyle\frac{d p^n}{d p^{n-1}}>\frac{d p^n}{d q^{n-1}}>\frac{d q^n}{d p^{n-1}}>\frac{d q^n}{d q^{n-1}}>\frac{d P^{n,2}}{d p^{n-1}}>\frac{d P^{n,1}}{d p^{n-1}}>\frac{d P^{n,2}}{d q^{n-1}}>\frac{d P^{n,1}}{d q^{n-1}}>\frac{d Q^{n,2}}{d p^{n-1}}>\frac{d Q^{n,1}}{d p^{n-1}}>\frac{d Q^{n,2}}{d q^{n-1}}>\frac{d Q^{n,1}}{d q^{n-1}}>\frac{d p^{n,2}}{d p^{n-1}}>\frac{d p^{n,1}}{d p^{n-1}}>\frac{d p^{n,2}}{d q^{n-1}}>\frac{d p^{n,1}}{d q^{n-1}}>\frac{d q^{n,2}}{d p^{n-1}}>\frac{d q^{n,1}}{d p^{n-1}}>\frac{d q^{n,2}}{d q^{n-1}}>\frac{d q^{n,1}}{d q^{n-1}}$, and the term order be lexicographical order. 

By utilizing \textsc{Maple} software, we compute the Gr\"obner basis of eqautions \eqref{31} and the normal form of  the left side of \eqref{32}. We write the Gr\"obner basis of eqautions \eqref{31} as $G$, and the normal form of  the left side of \eqref{32} w.r.t. $G$ as $g_{G}$. Because the number of indeterminates equals to the number of polynomials \eqref{31}, and all the polynomials are linear for indeterminates, according to the Gr\"obner basis theory, $g_{G}$ must have no any indeterminate, that is to say $g_{G}$ must be some formula about $a_{ij},b_i,\hat{a}_{ij},\hat{b}_i, \  i,j=1,2.$ We have $g_{G}$ of form:
\begin{equation}
\frac{Poly(a_{ij},b_i)}{Poly(a_{ij},b_i)+1}, \quad i,j=1,2,\label{md}
\end{equation}
where $Poly(a_{ij},b_i)$ represents some polynomials of coefficients $a_{ij},b_i,\hat{a}_{ij},\hat{b}_i, \  i,j=1,2,$ without content term. $g_{G}$ is pretty lengthiness and is showed in the appendix. 

Then by utilizing \textsc{Maple} software, we compute the Gr\"obner basis(written as $S$) of the left sides of symplectic conditions \eqref{sc1}-\eqref{sc2}, where the indeterminates order is $\hat b_2>\hat b_1>b_2>b_1>\hat a_{22}>\hat a_{21}>\hat a_{12}>\hat a_{11}>a_{22}>a_{21}>a_{12}>a_{11}$ and the term order is lexicographical order. Then, we compute the normal form of  the numerator and the dominator in \eqref{md} w.r.t. the above Gr\"obner basis $S$ respectively. We get  that the normal form of the numerator in \eqref{md} is $0$, and the normal form of the dominator in  \eqref{md} is $1$, which means 
\begin{equation}\label{end}
-\frac{d q^n}{d q^{n-1}}\frac{d p^n}{d p^{n-1}}+\frac{d p^n}{d q^{n-1}}\frac{d q^n}{d p^{n-1}} +1\xlongequal{reduced\ w.r.t.\ G}g_{G}\xlongequal{reduced\ w.r.t.\ S}\frac{0}{1}.
\end{equation}

According  to the definition \ref{defi1} and proposition \ref{pro1}, the normal form of a certain polynomial w.r.t. the Gr\"obner basis of polynomials set \eqref{31}(or \eqref{sc1}-\eqref{sc2}) is zero if and only if this polynomial belong to the ideal generated by \eqref{31}(or  \eqref{sc1}-\eqref{sc2}). Then, \eqref{end} means that $-\frac{d q^n}{d q^{n-1}}\frac{d p^n}{d p^{n-1}}+\frac{d p^n}{d q^{n-1}}\frac{d q^n}{d p^{n-1}} +1$ equals  $0$ after substituting \eqref{34} and  \eqref{sc1}-\eqref{sc2} into it. Thus, we prove that the PRK methods with symplectic conditions for determinate Hamiltonian system are symplectic.
\end{proof}

\begin{theorem}The stochastic s-stage PRK methods \eqref{spRK} with symplectic conditions \eqref{spRKsc} for stochastic Hamiltonian system \eqref{sHs} are symplectic.
\end{theorem}
\begin{proof}
Let
\begin{equation}\label{344}
\begin{aligned}
&P^{n,i}=H_q(p^{n,i},q^{n,i}),\qquad \tilde{P}^{n,i}=\tilde{H}_q(p^{n,i},q^{n,i}),\\
&Q^{n,i}=H_p(p^{n,i},q^{n,i}),\qquad \tilde{Q}^{n,i}=\tilde{H}_p(p^{n,i},q^{n,i}).
\end{aligned}
\end{equation}
From \eqref{344} and \eqref{spRK} we obtain\\

\begin{subequations}\label{34}
\begin{align}
&\frac{d p^n}{d p^{n-1}}=1-h\sum_{i=1}^s b_i \frac{d P^{n,i}}{d p^{n-1}}-\Delta B_n\sum_{i=1}^s \beta_{i} \frac{d \tilde{P}^{n,i}}{d p^{n-1}},\quad
\frac{d p^n}{d q^{n-1}}=-h\sum_{i=1}^s b_i \frac{d P^{n,i}}{d q^{n-1}}-\Delta B_n\sum_{i=1}^s \beta_{i} \frac{d \tilde{P}^{n,i}}{d q^{n-1}},\\
&\frac{d q^n}{d p^{n-1}}=h\sum_{i=1}^s \hat b_i \frac{d Q^{n,i}}{d p^{n-1}}+\Delta B_n\sum_{i=1}^s\hat{\beta}_i\frac{d \tilde{Q}^{n,i}}{d p^{n-1}},\quad
\frac{d q^n}{d q^{n-1}}=1+h\sum_{i=1}^s \hat b_i \frac{d Q^{n,i}}{d q^{n-1}}+\Delta B_n\sum_{i=1}^s\hat{\beta}_i\frac{d \tilde{Q}^{n,i}}{d q^{n-1}},\\
&\frac{d p^{n,i}}{d p^{n-1}}=1-h\sum_{j=1}^s a_{ij} \frac{d P^{n,j}}{d p^{n-1}}-\Delta B_n\sum_{j=1}^s \alpha_{ij}\frac{d \tilde{P}^{n,j}}{d p^{n-1}},\quad 
\frac{d p^{n,i}}{d q^{n-1}}=-h\sum_{j=1}^s a_{ij} \frac{d P^{n,j}}{d q^{n-1}}-\Delta B_n\sum_{j=1}^s \alpha_{ij}\frac{d \tilde{P}^{n,j}}{d q^{n-1}}\nonumber\\
&\qquad\qquad\qquad\qquad\qquad\qquad\qquad\qquad\qquad\qquad\qquad\qquad\qquad\qquad\qquad\qquad\qquad\qquad\quad i=1,2,\dots,s,\\
&\frac{d q^{n,i}}{d p^{n-1}}=h\sum_{j=1}^s \hat a_{ij} \frac{d Q^{n,j}}{d p^{n-1}}+\Delta B_n\sum_{j=1}^s \hat\alpha_{ij} \frac{d \tilde Q^{n,j}}{d p^{n-1}},\quad 
\frac{d q^{n,i}}{d q^{n-1}}=1+h\sum_{j=1}^s \hat a_{ij} \frac{d Q^{n,j}}{d q^{n-1}}+\Delta B_n\sum_{j=1}^s \hat\alpha_{ij} \frac{d \tilde Q^{n,j}}{d q^{n-1}}\nonumber\\
&\qquad\qquad\qquad\qquad\qquad\qquad\qquad\qquad\qquad\qquad\qquad\qquad\qquad\qquad\qquad\qquad\qquad\qquad\quad i=1,2,\dots,s,\\
&\frac{d P^{n,i}}{d p^{n-1}}=H_{qp}^i\frac{d p^{n,i}}{d p^{n-1}}+H_{qq}^i\frac{d q^{n,i}}{d p^{n-1}},\quad 
\frac{d P^{n,i}}{d q^{n-1}}=H_{qp}^i\frac{d p^{n,i}}{d q^{n-1}}+H_{qq}^i\frac{d q^{n,i}}{d q^{n-1}},\quad
\frac{d \tilde P^{n,i}}{d p^{n-1}}=\tilde H_{qp}^i\frac{d p^{n,i}}{d p^{n-1}}+\tilde H_{qq}^i\frac{d q^{n,i}}{d p^{n-1}},\nonumber\\
&\qquad\qquad\qquad\qquad\qquad\qquad\qquad\qquad\qquad\qquad\qquad\qquad\qquad\qquad\qquad\qquad\qquad\qquad\quad i=1,2,\dots,s,\\
&\frac{d \tilde P^{n,i}}{d q^{n-1}}=\tilde H_{qp}^i\frac{d p^{n,i}}{d q^{n-1}}+\tilde H_{qq}^i\frac{d q^{n,i}}{d q^{n-1}},\quad
\frac{d Q^{n,i}}{d p^{n-1}}=H_{pp}^i\frac{d p^{n,i}}{d p^{n-1}}+H_{pq}^i\frac{d q^{n,i}}{d p^{n-1}},\quad
\frac{d Q^{n,i}}{d q^{n-1}}=H_{pp}^i\frac{d p^{n,i}}{d q^{n-1}}+H_{pq}^i\frac{d q^{n,i}}{d q^{n-1}},\nonumber\\
&\qquad\qquad\qquad\qquad\qquad\qquad\qquad\qquad\qquad\qquad\qquad\qquad\qquad\qquad\qquad\qquad\qquad\qquad\quad i=1,2,\dots,s,\\
&\frac{d \tilde Q^{n,i}}{d p^{n-1}}=\tilde H_{pp}^i\frac{d p^{n,i}}{d p^{n-1}}+\tilde H_{pq}^i\frac{d q^{n,i}}{d p^{n-1}},\quad
\frac{d \tilde Q^{n,i}}{d q^{n-1}}=\tilde H_{pp}^i\frac{d p^{n,i}}{d q^{n-1}}+\tilde H_{pq}^i\frac{d q^{n,i}}{d q^{n-1}},\quad
\qquad\qquad\quad i=1,2,\dots,s.
\end{align}
\end{subequations}
where $H_{qp}(p^{n,i},q^{n,i})$ is written as $H_{qp}^i$ and another second partial derivatives of $H$ and $\tilde{H}$ are written in similar way. Equations \eqref{34} can be regarded as polynomials with indeterminates $\displaystyle\frac{d p^n}{d p^{n-1}},\frac{d p^n}{d q^{n-1}},\frac{d q^n}{d p^{n-1}},\frac{d q^n}{d q^{n-1}},\frac{d p^{n,i}}{d p^{n-1}}$, $\displaystyle\frac{d p^{n,i}}{d q^{n-1}},\frac{d q^{n,i}}{d p^{n-1}},\frac{d q^{n,i}}{d q^{n-1}},\frac{d P^{n,i}}{d p^{n-1}},\frac{d P^{n,i}}{d q^{n-1}},\frac{d Q^{n,i}}{d p^{n-1}},\frac{d Q^{n,i}}{d q^{n-1}},\frac{d \tilde P^{n,i}}{d p^{n-1}},\frac{d \tilde P^{n,i}}{d q^{n-1}},\frac{d \tilde Q^{n,i}}{d p^{n-1}},\frac{d \tilde Q^{n,i}}{d q^{n-1}}$. We take $s=2$ as an example. Let the indeterminates order be $\displaystyle\frac{d p^{n,2}}{d p^{n-1}}>\frac{d p^{n,1}}{d p^{n-1}}>\frac{d p^{n,2}}{d q^{n-1}}>\frac{d p^{n,1}}{d q^{n-1}}>\frac{d q^{n,2}}{d p^{n-1}}>\frac{d q^{n,1}}{d p^{n-1}}>\frac{d q^{n,2}}{d q^{n-1}}>\frac{d q^{n,1}}{d q^{n-1}}>\frac{d P^{n,1}}{d p^{n-1}}>\frac{d P^{n,2}}{d p^{n-1}}>\frac{d P^{n,1}}{d q^{n-1}}>\frac{d P^{n,2}}{d q^{n-1}}>\frac{d Q^{n,1}}{d p^{n-1}}>\frac{d Q^{n,2}}{d p^{n-1}}>\frac{d Q^{n,1}}{d q^{n-1}}>\frac{d Q^{n,2}}{d q^{n-1}}>\frac{d \tilde P^{n,1}}{d p^{n-1}}>\frac{d \tilde P^{n,2}}{d p^{n-1}}>\frac{d \tilde P^{n,1}}{d q^{n-1}}>\frac{d \tilde P^{n,2}}{d q^{n-1}}>\frac{d \tilde Q^{n,1}}{d p^{n-1}}>\frac{d \tilde Q^{n,2}}{d p^{n-1}}>\frac{d \tilde Q^{n,1}}{d q^{n-1}}>\frac{d \tilde Q^{n,2}}{d q^{n-1}}>\frac{d p^n}{d p^{n-1}}>\frac{d p^n}{d q^{n-1}}>\frac{d q^n}{d p^{n-1}}>\frac{d q^n}{d q^{n-1}}$, and the term order be lexicographical order. 

By utilizing \textsc{Maple} software, we compute the Gr\"obner basis of eqautions \eqref{34} and the normal form of  the left side of \eqref{32}. We write the Gr\"obner basis of eqautions \eqref{34} as $\bar{G}$, and the normal form of   the left side of \eqref{32} w.r.t.  $\bar{G}$ as $g_{\bar{G}}$.  We notice that the number of indeterminates equals to the number of polynomials in \eqref{31}, and all the polynomials are linear for indeterminates, for the same reason in the case for the determinate Hamiltonian systems, $g_{\bar G}$ must be some formula about $a_{ij},\alpha_{ij},b_i,\beta_{i},\hat{a}_{ij},\hat{\alpha}_{ij},\hat{b}_i,\hat{\beta}_{i}, \  i,j=1,2.$ We also have $g_{\bar{G}}$ of form:
\begin{equation}
\frac{Poly(a_{ij},\alpha_{ij},b_i,\beta_{i})}{Poly(a_{ij},\alpha_{ij},b_i,\beta_{i})+1}, \quad i,j=1,2,\label{smd}
\end{equation}
where $Poly(a_{ij},\alpha_{ij},b_i,\beta_{i})$ represents some polynomials of coefficients $a_{ij},\alpha_{ij},b_i,\beta_{i},\hat{a}_{ij},\hat{\alpha}_{ij},\hat{b}_i,\hat{\beta}_{i}, \  i,j=1,2,$ without content term. $g_{\bar{G}}$ is five times as long as $g_{G}$, so we don't show it in this paper.

Then by utilizing \textsc{Maple} software, we compute the Gr\"obner basis(written as $\bar{S}$) of the left sides of symplectic conditions \eqref{spRKsc}, where the indeterminates order is $\beta_2>b_2>\beta_1>b_1>\hat{\beta}_2>\hat b_2>\hat{\beta}_1>\hat b_1>\alpha_{22}>a_{22}>\alpha_{21}>a_{21}>\alpha_{12}>a_{12}>\alpha_{11}>a_{11}>\hat\alpha_{22}>\hat a_{22}>\hat\alpha_{21}>\hat a_{21}>\hat \alpha_{12}>\hat a_{12}>\hat \alpha_{11}>\hat a_{11}$ and the term order is lexicographical order. Then, we compute the normal form of  the numerator and the dominator in \eqref{smd} w.r.t. the above Gr\"obner basis $\bar{S}$ respectively. We get that the normal form of the numerator in \eqref{smd} is $0$, and the normal form of the dominator in  \eqref{smd} is $1$, which means 
\begin{equation}\label{end1}
-\frac{d q^n}{d q^{n-1}}\frac{d p^n}{d p^{n-1}}+\frac{d p^n}{d q^{n-1}}\frac{d q^n}{d p^{n-1}} +1\xlongequal{reduced\ w.r.t.\ \bar G}g_{\bar G}\xlongequal{reduced\ w.r.t.\ \bar S}0.
\end{equation}

According  to the definition \ref{defi1} and proposition \ref{pro1}, the normal form of a certain polynomial w.r.t. the Gr\"obner basis of set of polynomials \eqref{34}(or \eqref{spRKsc}) is zero if and only if this polynomial belong to the ideal generated by \eqref{34}(or \eqref{spRKsc}). Then, \eqref{end1} means that $-\frac{d q^n}{d q^{n-1}}\frac{d p^n}{d p^{n-1}}+\frac{d p^n}{d q^{n-1}}\frac{d q^n}{d p^{n-1}} +1$ equals $0$ after substituting \eqref{34} and \eqref{spRKsc} into it. Thus, we prove that the stochastic PRK methods with symplectic conditions for stochastic Hamiltonian system are symplectic.
 
\end{proof}

\section{Conclusions and Prospects}\hskip\parindent
In this paper, we respectively prove that the PRK methods with symplectic conditions for determinate Hamiltonian system and stochastic Hamiltonian system are symplectic by using Gr\"obner basis technology. By simple introduction of the Gr\"obner basis, we explain the function in the mechanical proving the symplecticity of PRK methods.

The future works are about the application of the Gr\"obner basis theory and the character set theory to the mechanical proving the structure-preservation laws for another numerical methods. As we all knows, in the common way, the energy conservation law or momentum conservation law of numerical methods are given by substituting numerical schemes into the discrete energy expression or discrete momentum expression. It is similar with in this paper for the symplectic structure conservation law. So we also can apply mechanization methods of theorem-proving to the proving the energy conservation law and momentum conservation law. It will help us to deal with complex or unproved conclusions.
 
\section{Acknowledgments}

\qquad This work is supported by Youth Foundation of Shandong Natural Science Foundation. (No. ZR2024QA201)

\appendixpage
\begin{equation*}
\begin{aligned}
g_G=&h(h^3H_{pp}^{1}H_{pp}^{2}H_{qq}^{1}H_{qq}^{2}a_{11}a_{22}\hat{a}_{11}\hat{b}_2 - h^3H_{pp}^{1}H_{pp}^{2}H_{qq}^{1}H_{qq}^{2}a_{11}a_{22}\hat{a}_{12}\hat{b}_1 - h^3H_{pp}^{1}H_{pp}^{2}H_{qq}^{1}H_{qq}^{2}a_{11}a_{22}\hat{a}_{21}\hat{b}_2\\
& + h^3H_{pp}^{1}H_{pp}^{2}H_{qq}^{1}H_{qq}^{2}a_{11}a_{22}\hat{a}_{22}\hat{b}_1 + h^3H_{pp}^{1}H_{pp}^{2}H_{qq}^{1}H_{qq}^{2}a_{11}\hat{a}_{11}\hat{a}_{22}b_2 - h^3H_{pp}^{1}H_{pp}^{2}H_{qq}^{1}H_{qq}^{2}a_{11}\hat{a}_{11}b_2\hat{b}_2 \\
&- h^3H_{pp}^{1}H_{pp}^{2}H_{qq}^{1}H_{qq}^{2}a_{11}\hat{a}_{12}\hat{a}_{21}b_2 + h^3H_{pp}^{1}H_{pp}^{2}H_{qq}^{1}H_{qq}^{2}a_{11}\hat{a}_{12}b_2\hat{b}_1 + h^3H_{pp}^{1}H_{pp}^{2}H_{qq}^{1}H_{qq}^{2}a_{11}\hat{a}_{21}b_2\hat{b}_2 \\
&- h^3H_{pp}^{1}H_{pp}^{2}H_{qq}^{1}H_{qq}^{2}a_{11}\hat{a}_{22}b_2\hat{b}_1 - h^3H_{pp}^{1}H_{pp}^{2}H_{qq}^{1}H_{qq}^{2}a_{12}a_{21}\hat{a}_{11}\hat{b}_2 + h^3H_{pp}^{1}H_{pp}^{2}H_{qq}^{1}H_{qq}^{2}a_{12}a_{21}\hat{a}_{12}\hat{b}_1\\
& + h^3H_{pp}^{1}H_{pp}^{2}H_{qq}^{1}H_{qq}^{2}a_{12}a_{21}\hat{a}_{21}\hat{b}_2 - h^3H_{pp}^{1}H_{pp}^{2}H_{qq}^{1}H_{qq}^{2}a_{12}a_{21}\hat{a}_{22}\hat{b}_1 - h^3H_{pp}^{1}H_{pp}^{2}H_{qq}^{1}H_{qq}^{2}a_{12}\hat{a}_{11}\hat{a}_{22}b_1\\
& + h^3H_{pp}^{1}H_{pp}^{2}H_{qq}^{1}H_{qq}^{2}a_{12}\hat{a}_{11}b_1\hat{b}_2 + h^3H_{pp}^{1}H_{pp}^{2}H_{qq}^{1}H_{qq}^{2}a_{12}\hat{a}_{12}\hat{a}_{21}b_1 - h^3H_{pp}^{1}H_{pp}^{2}H_{qq}^{1}H_{qq}^{2}a_{12}\hat{a}_{12}b_1\hat{b}_1\\
& - h^3H_{pp}^{1}H_{pp}^{2}H_{qq}^{1}H_{qq}^{2}a_{12}\hat{a}_{21}b_1\hat{b}_2 + h^3H_{pp}^{1}H_{pp}^{2}H_{qq}^{1}H_{qq}^{2}a_{12}\hat{a}_{22}b_1\hat{b}_1 - h^3H_{pp}^{1}H_{pp}^{2}H_{qq}^{1}H_{qq}^{2}a_{21}\hat{a}_{11}\hat{a}_{22}b_2 \\
&+ h^3H_{pp}^{1}H_{pp}^{2}H_{qq}^{1}H_{qq}^{2}a_{21}\hat{a}_{11}b_2\hat{b}_2 + h^3H_{pp}^{1}H_{pp}^{2}H_{qq}^{1}H_{qq}^{2}a_{21}\hat{a}_{12}\hat{a}_{21}b_2 - h^3H_{pp}^{1}H_{pp}^{2}H_{qq}^{1}H_{qq}^{2}a_{21}\hat{a}_{12}b_2\hat{b}_1\\
& - h^3H_{pp}^{1}H_{pp}^{2}H_{qq}^{1}H_{qq}^{2}a_{21}\hat{a}_{21}b_2\hat{b}_2 + h^3H_{pp}^{1}H_{pp}^{2}H_{qq}^{1}H_{qq}^{2}a_{21}\hat{a}_{22}b_2\hat{b}_1 + h^3H_{pp}^{1}H_{pp}^{2}H_{qq}^{1}H_{qq}^{2}a_{22}\hat{a}_{11}\hat{a}_{22}b_1\\
& - h^3H_{pp}^{1}H_{pp}^{2}H_{qq}^{1}H_{qq}^{2}a_{22}\hat{a}_{11}b_1\hat{b}_2 - h^3H_{pp}^{1}H_{pp}^{2}H_{qq}^{1}H_{qq}^{2}a_{22}\hat{a}_{12}\hat{a}_{21}b_1 + h^3H_{pp}^{1}H_{pp}^{2}H_{qq}^{1}H_{qq}^{2}a_{22}\hat{a}_{12}b_1\hat{b}_1 \\
&+ h^3H_{pp}^{1}H_{pp}^{2}H_{qq}^{1}H_{qq}^{2}a_{22}\hat{a}_{21}b_1\hat{b}_2 - h^3H_{pp}^{1}H_{pp}^{2}H_{qq}^{1}H_{qq}^{2}a_{22}\hat{a}_{22}b_1\hat{b}_1 - h^3H_{pp}^{1}H_{pq}^{2}H_{qp}^{2}H_{qq}^{1}a_{11}a_{22}\hat{a}_{11}\hat{b}_2\\
& + h^3H_{pp}^{1}H_{pq}^{2}H_{qp}^{2}H_{qq}^{1}a_{11}a_{22}\hat{a}_{12}\hat{b}_1 + h^3H_{pp}^{1}H_{pq}^{2}H_{qp}^{2}H_{qq}^{1}a_{11}a_{22}\hat{a}_{21}\hat{b}_2 - h^3H_{pp}^{1}H_{pq}^{2}H_{qp}^{2}H_{qq}^{1}a_{11}a_{22}\hat{a}_{22}\hat{b}_1\\
& - h^3H_{pp}^{1}H_{pq}^{2}H_{qp}^{2}H_{qq}^{1}a_{11}\hat{a}_{11}\hat{a}_{22}b_2 + h^3H_{pp}^{1}H_{pq}^{2}H_{qp}^{2}H_{qq}^{1}a_{11}\hat{a}_{11}b_2\hat{b}_2 + h^3H_{pp}^{1}H_{pq}^{2}H_{qp}^{2}H_{qq}^{1}a_{11}\hat{a}_{12}\hat{a}_{21}b_2 \\
&- h^3H_{pp}^{1}H_{pq}^{2}H_{qp}^{2}H_{qq}^{1}a_{11}\hat{a}_{12}b_2\hat{b}_1 - h^3H_{pp}^{1}H_{pq}^{2}H_{qp}^{2}H_{qq}^{1}a_{11}\hat{a}_{21}b_2\hat{b}_2 + h^3H_{pp}^{1}H_{pq}^{2}H_{qp}^{2}H_{qq}^{1}a_{11}\hat{a}_{22}b_2\hat{b}_1 \\
&+ h^3H_{pp}^{1}H_{pq}^{2}H_{qp}^{2}H_{qq}^{1}a_{12}a_{21}\hat{a}_{11}\hat{b}_2 - h^3H_{pp}^{1}H_{pq}^{2}H_{qp}^{2}H_{qq}^{1}a_{12}a_{21}\hat{a}_{12}\hat{b}_1 - h^3H_{pp}^{1}H_{pq}^{2}H_{qp}^{2}H_{qq}^{1}a_{12}a_{21}\hat{a}_{21}\hat{b}_2 \\
&+ h^3H_{pp}^{1}H_{pq}^{2}H_{qp}^{2}H_{qq}^{1}a_{12}a_{21}\hat{a}_{22}\hat{b}_1 + h^3H_{pp}^{1}H_{pq}^{2}H_{qp}^{2}H_{qq}^{1}a_{12}\hat{a}_{11}\hat{a}_{22}b_1 - h^3H_{pp}^{1}H_{pq}^{2}H_{qp}^{2}H_{qq}^{1}a_{12}\hat{a}_{11}b_1\hat{b}_2\\
& - h^3H_{pp}^{1}H_{pq}^{2}H_{qp}^{2}H_{qq}^{1}a_{12}\hat{a}_{12}\hat{a}_{21}b_1 + h^3H_{pp}^{1}H_{pq}^{2}H_{qp}^{2}H_{qq}^{1}a_{12}\hat{a}_{12}b_1\hat{b}_1 + h^3H_{pp}^{1}H_{pq}^{2}H_{qp}^{2}H_{qq}^{1}a_{12}\hat{a}_{21}b_1\hat{b}_2 \\
&- h^3H_{pp}^{1}H_{pq}^{2}H_{qp}^{2}H_{qq}^{1}a_{12}\hat{a}_{22}b_1\hat{b}_1 + h^3H_{pp}^{1}H_{pq}^{2}H_{qp}^{2}H_{qq}^{1}a_{21}\hat{a}_{11}\hat{a}_{22}b_2 - h^3H_{pp}^{1}H_{pq}^{2}H_{qp}^{2}H_{qq}^{1}a_{21}\hat{a}_{11}b_2\hat{b}_2 \\
&- h^3H_{pp}^{1}H_{pq}^{2}H_{qp}^{2}H_{qq}^{1}a_{21}\hat{a}_{12}\hat{a}_{21}b_2 + h^3H_{pp}^{1}H_{pq}^{2}H_{qp}^{2}H_{qq}^{1}a_{21}\hat{a}_{12}b_2\hat{b}_1 + h^3H_{pp}^{1}H_{pq}^{2}H_{qp}^{2}H_{qq}^{1}a_{21}\hat{a}_{21}b_2\hat{b}_2 \\
&- h^3H_{pp}^{1}H_{pq}^{2}H_{qp}^{2}H_{qq}^{1}a_{21}\hat{a}_{22}b_2\hat{b}_1 - h^3H_{pp}^{1}H_{pq}^{2}H_{qp}^{2}H_{qq}^{1}a_{22}\hat{a}_{11}\hat{a}_{22}b_1 + h^3H_{pp}^{1}H_{pq}^{2}H_{qp}^{2}H_{qq}^{1}a_{22}\hat{a}_{11}b_1\hat{b}_2 \\
&+ h^3H_{pp}^{1}H_{pq}^{2}H_{qp}^{2}H_{qq}^{1}a_{22}\hat{a}_{12}\hat{a}_{21}b_1 - h^3H_{pp}^{1}H_{pq}^{2}H_{qp}^{2}H_{qq}^{1}a_{22}\hat{a}_{12}b_1\hat{b}_1 - h^3H_{pp}^{1}H_{pq}^{2}H_{qp}^{2}H_{qq}^{1}a_{22}\hat{a}_{21}b_1\hat{b}_2 \\
&+ h^3H_{pp}^{1}H_{pq}^{2}H_{qp}^{2}H_{qq}^{1}a_{22}\hat{a}_{22}b_1\hat{b}_1 - h^3H_{pp}^{2}H_{pq}^{1}H_{qp}^{1}H_{qq}^{2}a_{11}a_{22}\hat{a}_{11}\hat{b}_2 + h^3H_{pp}^{2}H_{pq}^{1}H_{qp}^{1}H_{qq}^{2}a_{11}a_{22}\hat{a}_{12}\hat{b}_1 \\
&+ h^3H_{pp}^{2}H_{pq}^{1}H_{qp}^{1}H_{qq}^{2}a_{11}a_{22}\hat{a}_{21}\hat{b}_2 - h^3H_{pp}^{2}H_{pq}^{1}H_{qp}^{1}H_{qq}^{2}a_{11}a_{22}\hat{a}_{22}\hat{b}_1 - h^3H_{pp}^{2}H_{pq}^{1}H_{qp}^{1}H_{qq}^{2}a_{11}\hat{a}_{11}\hat{a}_{22}b_2\\
& + h^3H_{pp}^{2}H_{pq}^{1}H_{qp}^{1}H_{qq}^{2}a_{11}\hat{a}_{11}b_2\hat{b}_2 + h^3H_{pp}^{2}H_{pq}^{1}H_{qp}^{1}H_{qq}^{2}a_{11}\hat{a}_{12}\hat{a}_{21}b_2 - h^3H_{pp}^{2}H_{pq}^{1}H_{qp}^{1}H_{qq}^{2}a_{11}\hat{a}_{12}b_2\hat{b}_1 \\
&- h^3H_{pp}^{2}H_{pq}^{1}H_{qp}^{1}H_{qq}^{2}a_{11}\hat{a}_{21}b_2\hat{b}_2 + h^3H_{pp}^{2}H_{pq}^{1}H_{qp}^{1}H_{qq}^{2}a_{11}\hat{a}_{22}b_2\hat{b}_1 + h^3H_{pp}^{2}H_{pq}^{1}H_{qp}^{1}H_{qq}^{2}a_{12}a_{21}\hat{a}_{11}\hat{b}_2\\
& - h^3H_{pp}^{2}H_{pq}^{1}H_{qp}^{1}H_{qq}^{2}a_{12}a_{21}\hat{a}_{12}\hat{b}_1 - h^3H_{pp}^{2}H_{pq}^{1}H_{qp}^{1}H_{qq}^{2}a_{12}a_{21}\hat{a}_{21}\hat{b}_2 + h^3H_{pp}^{2}H_{pq}^{1}H_{qp}^{1}H_{qq}^{2}a_{12}a_{21}\hat{a}_{22}\hat{b}_1 \\
&+ h^3H_{pp}^{2}H_{pq}^{1}H_{qp}^{1}H_{qq}^{2}a_{12}\hat{a}_{11}\hat{a}_{22}b_1 - h^3H_{pp}^{2}H_{pq}^{1}H_{qp}^{1}H_{qq}^{2}a_{12}\hat{a}_{11}b_1\hat{b}_2 - h^3H_{pp}^{2}H_{pq}^{1}H_{qp}^{1}H_{qq}^{2}a_{12}\hat{a}_{12}\hat{a}_{21}b_1 \\
&+ h^3H_{pp}^{2}H_{pq}^{1}H_{qp}^{1}H_{qq}^{2}a_{12}\hat{a}_{12}b_1\hat{b}_1 + h^3H_{pp}^{2}H_{pq}^{1}H_{qp}^{1}H_{qq}^{2}a_{12}\hat{a}_{21}b_1\hat{b}_2 - h^3H_{pp}^{2}H_{pq}^{1}H_{qp}^{1}H_{qq}^{2}a_{12}\hat{a}_{22}b_1\hat{b}_1 \\
&+ h^3H_{pp}^{2}H_{pq}^{1}H_{qp}^{1}H_{qq}^{2}a_{21}\hat{a}_{11}\hat{a}_{22}b_2 - h^3H_{pp}^{2}H_{pq}^{1}H_{qp}^{1}H_{qq}^{2}a_{21}\hat{a}_{11}b_2\hat{b}_2 - h^3H_{pp}^{2}H_{pq}^{1}H_{qp}^{1}H_{qq}^{2}a_{21}\hat{a}_{12}\hat{a}_{21}b_2 
\end{aligned}
\end{equation*}

\begin{equation*}
\begin{aligned}
&+ h^3H_{pp}^{2}H_{pq}^{1}H_{qp}^{1}H_{qq}^{2}a_{21}\hat{a}_{12}b_2\hat{b}_1 + h^3H_{pp}^{2}H_{pq}^{1}H_{qp}^{1}H_{qq}^{2}a_{21}\hat{a}_{21}b_2\hat{b}_2 - h^3H_{pp}^{2}H_{pq}^{1}H_{qp}^{1}H_{qq}^{2}a_{21}\hat{a}_{22}b_2\hat{b}_1 \\
&- h^3H_{pp}^{2}H_{pq}^{1}H_{qp}^{1}H_{qq}^{2}a_{22}\hat{a}_{11}\hat{a}_{22}b_1 + h^3H_{pp}^{2}H_{pq}^{1}H_{qp}^{1}H_{qq}^{2}a_{22}\hat{a}_{11}b_1\hat{b}_2 + h^3H_{pp}^{2}H_{pq}^{1}H_{qp}^{1}H_{qq}^{2}a_{22}\hat{a}_{12}\hat{a}_{21}b_1\\
& - h^3H_{pp}^{2}H_{pq}^{1}H_{qp}^{1}H_{qq}^{2}a_{22}\hat{a}_{12}b_1\hat{b}_1 - h^3H_{pp}^{2}H_{pq}^{1}H_{qp}^{1}H_{qq}^{2}a_{22}\hat{a}_{21}b_1\hat{b}_2 + h^3H_{pp}^{2}H_{pq}^{1}H_{qp}^{1}H_{qq}^{2}a_{22}\hat{a}_{22}b_1\hat{b}_1 \\
&+ h^3H_{pq}^{1}H_{pq}^{2}H_{qp}^{1}H_{qp}^{2}a_{11}a_{22}\hat{a}_{11}\hat{b}_2 - h^3H_{pq}^{1}H_{pq}^{2}H_{qp}^{1}H_{qp}^{2}a_{11}a_{22}\hat{a}_{12}\hat{b}_1 - h^3H_{pq}^{1}H_{pq}^{2}H_{qp}^{1}H_{qp}^{2}a_{11}a_{22}\hat{a}_{21}\hat{b}_2 \\
&+ h^3H_{pq}^{1}H_{pq}^{2}H_{qp}^{1}H_{qp}^{2}a_{11}a_{22}\hat{a}_{22}\hat{b}_1 + h^3H_{pq}^{1}H_{pq}^{2}H_{qp}^{1}H_{qp}^{2}a_{11}\hat{a}_{11}\hat{a}_{22}b_2 - h^3H_{pq}^{1}H_{pq}^{2}H_{qp}^{1}H_{qp}^{2}a_{11}\hat{a}_{11}b_2\hat{b}_2 \\
&- h^3H_{pq}^{1}H_{pq}^{2}H_{qp}^{1}H_{qp}^{2}a_{11}\hat{a}_{12}\hat{a}_{21}b_2 + h^3H_{pq}^{1}H_{pq}^{2}H_{qp}^{1}H_{qp}^{2}a_{11}\hat{a}_{12}b_2\hat{b}_1 + h^3H_{pq}^{1}H_{pq}^{2}H_{qp}^{1}H_{qp}^{2}a_{11}\hat{a}_{21}b_2\hat{b}_2 \\
&- h^3H_{pq}^{1}H_{pq}^{2}H_{qp}^{1}H_{qp}^{2}a_{11}\hat{a}_{22}b_2\hat{b}_1 - h^3H_{pq}^{1}H_{pq}^{2}H_{qp}^{1}H_{qp}^{2}a_{12}a_{21}\hat{a}_{11}\hat{b}_2 + h^3H_{pq}^{1}H_{pq}^{2}H_{qp}^{1}H_{qp}^{2}a_{12}a_{21}\hat{a}_{12}\hat{b}_1 \\
&+ h^3H_{pq}^{1}H_{pq}^{2}H_{qp}^{1}H_{qp}^{2}a_{12}a_{21}\hat{a}_{21}\hat{b}_2 - h^3H_{pq}^{1}H_{pq}^{2}H_{qp}^{1}H_{qp}^{2}a_{12}a_{21}\hat{a}_{22}\hat{b}_1 - h^3H_{pq}^{1}H_{pq}^{2}H_{qp}^{1}H_{qp}^{2}a_{12}\hat{a}_{11}\hat{a}_{22}b_1 \\
&+ h^3H_{pq}^{1}H_{pq}^{2}H_{qp}^{1}H_{qp}^{2}a_{12}\hat{a}_{11}b_1\hat{b}_2 + h^3H_{pq}^{1}H_{pq}^{2}H_{qp}^{1}H_{qp}^{2}a_{12}\hat{a}_{12}\hat{a}_{21}b_1 - h^3H_{pq}^{1}H_{pq}^{2}H_{qp}^{1}H_{qp}^{2}a_{12}\hat{a}_{12}b_1\hat{b}_1 \\
&- h^3H_{pq}^{1}H_{pq}^{2}H_{qp}^{1}H_{qp}^{2}a_{12}\hat{a}_{21}b_1\hat{b}_2 + h^3H_{pq}^{1}H_{pq}^{2}H_{qp}^{1}H_{qp}^{2}a_{12}\hat{a}_{22}b_1\hat{b}_1 - h^3H_{pq}^{1}H_{pq}^{2}H_{qp}^{1}H_{qp}^{2}a_{21}\hat{a}_{11}\hat{a}_{22}b_2 \\
&+ h^3H_{pq}^{1}H_{pq}^{2}H_{qp}^{1}H_{qp}^{2}a_{21}\hat{a}_{11}b_2\hat{b}_2 + h^3H_{pq}^{1}H_{pq}^{2}H_{qp}^{1}H_{qp}^{2}a_{21}\hat{a}_{12}\hat{a}_{21}b_2 - h^3H_{pq}^{1}H_{pq}^{2}H_{qp}^{1}H_{qp}^{2}a_{21}\hat{a}_{12}b_2\hat{b}_1 \\
&- h^3H_{pq}^{1}H_{pq}^{2}H_{qp}^{1}H_{qp}^{2}a_{21}\hat{a}_{21}b_2\hat{b}_2 + h^3H_{pq}^{1}H_{pq}^{2}H_{qp}^{1}H_{qp}^{2}a_{21}\hat{a}_{22}b_2\hat{b}_1 + h^3H_{pq}^{1}H_{pq}^{2}H_{qp}^{1}H_{qp}^{2}a_{22}\hat{a}_{11}\hat{a}_{22}b_1 \\
&- h^3H_{pq}^{1}H_{pq}^{2}H_{qp}^{1}H_{qp}^{2}a_{22}\hat{a}_{11}b_1\hat{b}_2 - h^3H_{pq}^{1}H_{pq}^{2}H_{qp}^{1}H_{qp}^{2}a_{22}\hat{a}_{12}\hat{a}_{21}b_1 + h^3H_{pq}^{1}H_{pq}^{2}H_{qp}^{1}H_{qp}^{2}a_{22}\hat{a}_{12}b_1\hat{b}_1 \\
&+ h^3H_{pq}^{1}H_{pq}^{2}H_{qp}^{1}H_{qp}^{2}a_{22}\hat{a}_{21}b_1\hat{b}_2 - h^3H_{pq}^{1}H_{pq}^{2}H_{qp}^{1}H_{qp}^{2}a_{22}\hat{a}_{22}b_1\hat{b}_1 - h^2H_{pp}^{1}H_{pq}^{2}H_{qq}^{1}a_{11}\hat{a}_{11}\hat{b}_2 \\
&+ h^2H_{pp}^{1}H_{pq}^{2}H_{qq}^{1}a_{11}\hat{a}_{12}\hat{b}_1 + h^2H_{pp}^{1}H_{pq}^{2}H_{qq}^{1}a_{11}\hat{a}_{21}\hat{b}_2 - h^2H_{pp}^{1}H_{pq}^{2}H_{qq}^{1}a_{11}\hat{a}_{22}\hat{b}_1 - h^2H_{pp}^{1}H_{pq}^{2}H_{qq}^{1}\hat{a}_{11}\hat{a}_{22}b_1 \\
&+ h^2H_{pp}^{1}H_{pq}^{2}H_{qq}^{1}\hat{a}_{11}b_1\hat{b}_2 + h^2H_{pp}^{1}H_{pq}^{2}H_{qq}^{1}\hat{a}_{12}\hat{a}_{21}b_1 - h^2H_{pp}^{1}H_{pq}^{2}H_{qq}^{1}\hat{a}_{12}b_1\hat{b}_1 - h^2H_{pp}^{1}H_{pq}^{2}H_{qq}^{1}\hat{a}_{21}b_1\hat{b}_2 \\
&+ h^2H_{pp}^{1}H_{pq}^{2}H_{qq}^{1}\hat{a}_{22}b_1\hat{b}_1 + h^2H_{pp}^{1}H_{qp}^{2}H_{qq}^{1}a_{11}a_{22}\hat{b}_1 + h^2H_{pp}^{1}H_{qp}^{2}H_{qq}^{1}a_{11}\hat{a}_{11}b_2 - h^2H_{pp}^{1}H_{qp}^{2}H_{qq}^{1}a_{11}b_2\hat{b}_1 \\
&- h^2H_{pp}^{1}H_{qp}^{2}H_{qq}^{1}a_{12}a_{21}\hat{b}_1 - h^2H_{pp}^{1}H_{qp}^{2}H_{qq}^{1}a_{12}\hat{a}_{11}b_1 + h^2H_{pp}^{1}H_{qp}^{2}H_{qq}^{1}a_{12}b_1\hat{b}_1 - h^2H_{pp}^{1}H_{qp}^{2}H_{qq}^{1}a_{21}\hat{a}_{11}b_2 \\
&+ h^2H_{pp}^{1}H_{qp}^{2}H_{qq}^{1}a_{21}b_2\hat{b}_1 + h^2H_{pp}^{1}H_{qp}^{2}H_{qq}^{1}a_{22}\hat{a}_{11}b_1 - h^2H_{pp}^{1}H_{qp}^{2}H_{qq}^{1}a_{22}b_1\hat{b}_1 - h^2H_{pp}^{2}H_{pq}^{1}H_{qq}^{2}a_{22}\hat{a}_{11}\hat{b}_2\\
& + h^2H_{pp}^{2}H_{pq}^{1}H_{qq}^{2}a_{22}\hat{a}_{12}\hat{b}_1 + h^2H_{pp}^{2}H_{pq}^{1}H_{qq}^{2}a_{22}\hat{a}_{21}\hat{b}_2 - h^2H_{pp}^{2}H_{pq}^{1}H_{qq}^{2}a_{22}\hat{a}_{22}\hat{b}_1 - h^2H_{pp}^{2}H_{pq}^{1}H_{qq}^{2}\hat{a}_{11}\hat{a}_{22}b_2 \\
&+ h^2H_{pp}^{2}H_{pq}^{1}H_{qq}^{2}\hat{a}_{11}b_2\hat{b}_2 + h^2H_{pp}^{2}H_{pq}^{1}H_{qq}^{2}\hat{a}_{12}\hat{a}_{21}b_2 - h^2H_{pp}^{2}H_{pq}^{1}H_{qq}^{2}\hat{a}_{12}b_2\hat{b}_1 - h^2H_{pp}^{2}H_{pq}^{1}H_{qq}^{2}\hat{a}_{21}b_2\hat{b}_2 \\
&+ h^2H_{pp}^{2}H_{pq}^{1}H_{qq}^{2}\hat{a}_{22}b_2\hat{b}_1 + h^2H_{pp}^{2}H_{qp}^{1}H_{qq}^{2}a_{11}a_{22}\hat{b}_2 + h^2H_{pp}^{2}H_{qp}^{1}H_{qq}^{2}a_{11}\hat{a}_{22}b_2 - h^2H_{pp}^{2}H_{qp}^{1}H_{qq}^{2}a_{11}b_2\hat{b}_2 \\
&- h^2H_{pp}^{2}H_{qp}^{1}H_{qq}^{2}a_{12}a_{21}\hat{b}_2 - h^2H_{pp}^{2}H_{qp}^{1}H_{qq}^{2}a_{12}\hat{a}_{22}b_1 + h^2H_{pp}^{2}H_{qp}^{1}H_{qq}^{2}a_{12}b_1\hat{b}_2 - h^2H_{pp}^{2}H_{qp}^{1}H_{qq}^{2}a_{21}\hat{a}_{22}b_2 \\
&+ h^2H_{pp}^{2}H_{qp}^{1}H_{qq}^{2}a_{21}b_2\hat{b}_2 + h^2H_{pp}^{2}H_{qp}^{1}H_{qq}^{2}a_{22}\hat{a}_{22}b_1 - h^2H_{pp}^{2}H_{qp}^{1}H_{qq}^{2}a_{22}b_1\hat{b}_2 + h^2H_{pq}^{1}H_{pq}^{2}H_{qp}^{1}a_{11}\hat{a}_{11}\hat{b}_2 \\
&- h^2H_{pq}^{1}H_{pq}^{2}H_{qp}^{1}a_{11}\hat{a}_{12}\hat{b}_1 - h^2H_{pq}^{1}H_{pq}^{2}H_{qp}^{1}a_{11}\hat{a}_{21}\hat{b}_2 + h^2H_{pq}^{1}H_{pq}^{2}H_{qp}^{1}a_{11}\hat{a}_{22}\hat{b}_1 + h^2H_{pq}^{1}H_{pq}^{2}H_{qp}^{1}\hat{a}_{11}\hat{a}_{22}b_1\\
& - h^2H_{pq}^{1}H_{pq}^{2}H_{qp}^{1}\hat{a}_{11}b_1\hat{b}_2 - h^2H_{pq}^{1}H_{pq}^{2}H_{qp}^{1}\hat{a}_{12}\hat{a}_{21}b_1 + h^2H_{pq}^{1}H_{pq}^{2}H_{qp}^{1}\hat{a}_{12}b_1\hat{b}_1 + h^2H_{pq}^{1}H_{pq}^{2}H_{qp}^{1}\hat{a}_{21}b_1\hat{b}_2 \\
&- h^2H_{pq}^{1}H_{pq}^{2}H_{qp}^{1}\hat{a}_{22}b_1\hat{b}_1 + h^2H_{pq}^{1}H_{pq}^{2}H_{qp}^{2}a_{22}\hat{a}_{11}\hat{b}_2 - h^2H_{pq}^{1}H_{pq}^{2}H_{qp}^{2}a_{22}\hat{a}_{12}\hat{b}_1 - h^2H_{pq}^{1}H_{pq}^{2}H_{qp}^{2}a_{22}\hat{a}_{21}\hat{b}_2\\
& + h^2H_{pq}^{1}H_{pq}^{2}H_{qp}^{2}a_{22}\hat{a}_{22}\hat{b}_1 + h^2H_{pq}^{1}H_{pq}^{2}H_{qp}^{2}\hat{a}_{11}\hat{a}_{22}b_2 - h^2H_{pq}^{1}H_{pq}^{2}H_{qp}^{2}\hat{a}_{11}b_2\hat{b}_2 - h^2H_{pq}^{1}H_{pq}^{2}H_{qp}^{2}\hat{a}_{12}\hat{a}_{21}b_2 \\
&+ h^2H_{pq}^{1}H_{pq}^{2}H_{qp}^{2}\hat{a}_{12}b_2\hat{b}_1 + h^2H_{pq}^{1}H_{pq}^{2}H_{qp}^{2}\hat{a}_{21}b_2\hat{b}_2 - h^2H_{pq}^{1}H_{pq}^{2}H_{qp}^{2}\hat{a}_{22}b_2\hat{b}_1 - h^2H_{pq}^{1}H_{qp}^{1}H_{qp}^{2}a_{11}a_{22}\hat{b}_1 \\
\end{aligned}
\end{equation*}

\begin{equation*}
\begin{aligned}
&- h^2H_{pq}^{1}H_{qp}^{1}H_{qp}^{2}a_{11}\hat{a}_{11}b_2 + h^2H_{pq}^{1}H_{qp}^{1}H_{qp}^{2}a_{11}b_2\hat{b}_1 + h^2H_{pq}^{1}H_{qp}^{1}H_{qp}^{2}a_{12}a_{21}\hat{b}_1 + h^2H_{pq}^{1}H_{qp}^{1}H_{qp}^{2}a_{12}\hat{a}_{11}b_1\\
& - h^2H_{pq}^{1}H_{qp}^{1}H_{qp}^{2}a_{12}b_1\hat{b}_1 + h^2H_{pq}^{1}H_{qp}^{1}H_{qp}^{2}a_{21}\hat{a}_{11}b_2 - h^2H_{pq}^{1}H_{qp}^{1}H_{qp}^{2}a_{21}b_2\hat{b}_1 - h^2H_{pq}^{1}H_{qp}^{1}H_{qp}^{2}a_{22}\hat{a}_{11}b_1\\
& + h^2H_{pq}^{1}H_{qp}^{1}H_{qp}^{2}a_{22}b_1\hat{b}_1 - h^2H_{pq}^{2}H_{qp}^{1}H_{qp}^{2}a_{11}a_{22}\hat{b}_2 - h^2H_{pq}^{2}H_{qp}^{1}H_{qp}^{2}a_{11}\hat{a}_{22}b_2 + h^2H_{pq}^{2}H_{qp}^{1}H_{qp}^{2}a_{11}b_2\hat{b}_2 \\
&+ h^2H_{pq}^{2}H_{qp}^{1}H_{qp}^{2}a_{12}a_{21}\hat{b}_2 + h^2H_{pq}^{2}H_{qp}^{1}H_{qp}^{2}a_{12}\hat{a}_{22}b_1 - h^2H_{pq}^{2}H_{qp}^{1}H_{qp}^{2}a_{12}b_1\hat{b}_2 + h^2H_{pq}^{2}H_{qp}^{1}H_{qp}^{2}a_{21}\hat{a}_{22}b_2\\
& - h^2H_{pq}^{2}H_{qp}^{1}H_{qp}^{2}a_{21}b_2\hat{b}_2 - h^2H_{pq}^{2}H_{qp}^{1}H_{qp}^{2}a_{22}\hat{a}_{22}b_1 + h^2H_{pq}^{2}H_{qp}^{1}H_{qp}^{2}a_{22}b_1\hat{b}_2 + hH_{pp}^{1}H_{qq}^{1}a_{11}\hat{b}_1  \\
&- hH_{pp}^{1}H_{qq}^{1}b_1\hat{b}_1 + hH_{pp}^{1}H_{qq}^{2}a_{12}\hat{b}_1 + hH_{pp}^{1}H_{qq}^{2}\hat{a}_{21}b_2 - hH_{pp}^{1}H_{qq}^{2}b_2\hat{b}_1 + hH_{pp}^{2}H_{qq}^{1}a_{21}\hat{b}_2 + hH_{pp}^{2}H_{qq}^{1}\hat{a}_{12}b_1 \\
&- hH_{pp}^{2}H_{qq}^{1}b_1\hat{b}_2 + hH_{pp}^{2}H_{qq}^{2}a_{22}\hat{b}_2 + hH_{pp}^{2}H_{qq}^{2}\hat{a}_{22}b_2 - hH_{pp}^{2}H_{qq}^{2}b_2\hat{b}_2 + hH_{pq}^{1}H_{pq}^{2}\hat{a}_{11}\hat{b}_2 - hH_{pq}^{1}H_{pq}^{2}\hat{a}_{12}\hat{b}_1\\
& - hH_{pq}^{1}H_{pq}^{2}\hat{a}_{21}\hat{b}_2 + hH_{pq}^{1}H_{pq}^{2}\hat{a}_{22}\hat{b}_1 - hH_{pq}^{1}H_{qp}^{1}a_{11}\hat{b}_1 - hH_{pq}^{1}H_{qp}^{1}\hat{a}_{11}b_1 + hH_{pq}^{1}H_{qp}^{1}b_1\hat{b}_1 - hH_{pq}^{1}H_{qp}^{2}a_{22}\hat{b}_1\\
& - hH_{pq}^{1}H_{qp}^{2}\hat{a}_{11}b_2 + hH_{pq}^{1}H_{qp}^{2}b_2\hat{b}_1 - hH_{pq}^{2}H_{qp}^{1}a_{11}\hat{b}_2 - hH_{pq}^{2}H_{qp}^{1}\hat{a}_{22}b_1 + hH_{pq}^{2}H_{qp}^{1}b_1\hat{b}_2 - hH_{pq}^{2}H_{qp}^{2}a_{22}\hat{b}_2\\
& - hH_{pq}^{2}H_{qp}^{2}\hat{a}_{22}b_2 + hH_{pq}^{2}H_{qp}^{2}b_2\hat{b}_2 + hH_{qp}^{1}H_{qp}^{2}a_{11}b_2 - hH_{qp}^{1}H_{qp}^{2}a_{12}b_1 - hH_{qp}^{1}H_{qp}^{2}a_{21}b_2 + hH_{qp}^{1}H_{qp}^{2}a_{22}b_1 \\
&- H_{pq}^{1}\hat{b}_1 - H_{pq}^{2}\hat{b}_2 + H_{qp}^{1}b_1 + H_{qp}^{2}b_2+ hH_{pp}^{1}H_{qq}^{1}\hat{a}_{11}b_1)\\
&/(h^4H_{pp}^{1}H_{pp}^{2}H_{qq}^{1}H_{qq}^{2}a_{11}a_{22}\hat{a}_{11}\hat{a}_{22} - h^4H_{pp}^{1}H_{pp}^{2}H_{qq}^{1}H_{qq}^{2}a_{11}a_{22}\hat{a}_{12}\hat{a}_{21} - h^4H_{pp}^{1}H_{pp}^{2}H_{qq}^{1}H_{qq}^{2}a_{12}a_{21}\hat{a}_{11}\hat{a}_{22} \\
&+ h^4H_{pp}^{1}H_{pp}^{2}H_{qq}^{1}H_{qq}^{2}a_{12}a_{21}\hat{a}_{12}\hat{a}_{21} - h^4H_{pp}^{1}H_{pq}^{2}H_{qp}^{2}H_{qq}^{1}a_{11}a_{22}\hat{a}_{11}\hat{a}_{22} + h^4H_{pp}^{1}H_{pq}^{2}H_{qp}^{2}H_{qq}^{1}a_{11}a_{22}\hat{a}_{12}\hat{a}_{21} \\
&+ h^4H_{pp}^{1}H_{pq}^{2}H_{qp}^{2}H_{qq}^{1}a_{12}a_{21}\hat{a}_{11}\hat{a}_{22} - h^4H_{pp}^{1}H_{pq}^{2}H_{qp}^{2}H_{qq}^{1}a_{12}a_{21}\hat{a}_{12}\hat{a}_{21} - h^4H_{pp}^{2}H_{pq}^{1}H_{qp}^{1}H_{qq}^{2}a_{11}a_{22}\hat{a}_{11}\hat{a}_{22} \\
&+ h^4H_{pp}^{2}H_{pq}^{1}H_{qp}^{1}H_{qq}^{2}a_{11}a_{22}\hat{a}_{12}\hat{a}_{21} + h^4H_{pp}^{2}H_{pq}^{1}H_{qp}^{1}H_{qq}^{2}a_{12}a_{21}\hat{a}_{11}\hat{a}_{22} - h^4H_{pp}^{2}H_{pq}^{1}H_{qp}^{1}H_{qq}^{2}a_{12}a_{21}\hat{a}_{12}\hat{a}_{21} \\
&+ h^4H_{pq}^{1}H_{pq}^{2}H_{qp}^{1}H_{qp}^{2}a_{11}a_{22}\hat{a}_{11}\hat{a}_{22} - h^4H_{pq}^{1}H_{pq}^{2}H_{qp}^{1}H_{qp}^{2}a_{11}a_{22}\hat{a}_{12}\hat{a}_{21} - h^4H_{pq}^{1}H_{pq}^{2}H_{qp}^{1}H_{qp}^{2}a_{12}a_{21}\hat{a}_{11}\hat{a}_{22} \\
&+ h^4H_{pq}^{1}H_{pq}^{2}H_{qp}^{1}H_{qp}^{2}a_{12}a_{21}\hat{a}_{12}\hat{a}_{21} - h^3H_{pp}^{1}H_{pq}^{2}H_{qq}^{1}a_{11}\hat{a}_{11}\hat{a}_{22} + h^3H_{pp}^{1}H_{pq}^{2}H_{qq}^{1}a_{11}\hat{a}_{12}\hat{a}_{21} \\
&+ h^3H_{pp}^{1}H_{qp}^{2}H_{qq}^{1}a_{11}a_{22}\hat{a}_{11} - h^3H_{pp}^{1}H_{qp}^{2}H_{qq}^{1}a_{12}a_{21}\hat{a}_{11} - h^3H_{pp}^{2}H_{pq}^{1}H_{qq}^{2}a_{22}\hat{a}_{11}\hat{a}_{22} + h^3H_{pp}^{2}H_{pq}^{1}H_{qq}^{2}a_{22}\hat{a}_{12}\hat{a}_{21} \\
&+ h^3H_{pp}^{2}H_{qp}^{1}H_{qq}^{2}a_{11}a_{22}\hat{a}_{22} - h^3H_{pp}^{2}H_{qp}^{1}H_{qq}^{2}a_{12}a_{21}\hat{a}_{22} + h^3H_{pq}^{1}H_{pq}^{2}H_{qp}^{1}a_{11}\hat{a}_{11}\hat{a}_{22} - h^3H_{pq}^{1}H_{pq}^{2}H_{qp}^{1}a_{11}\hat{a}_{12}\hat{a}_{21}\\
& + h^3H_{pq}^{1}H_{pq}^{2}H_{qp}^{2}a_{22}\hat{a}_{11}\hat{a}_{22} - h^3H_{pq}^{1}H_{pq}^{2}H_{qp}^{2}a_{22}\hat{a}_{12}\hat{a}_{21} - h^3H_{pq}^{1}H_{qp}^{1}H_{qp}^{2}a_{11}a_{22}\hat{a}_{11} + h^3H_{pq}^{1}H_{qp}^{1}H_{qp}^{2}a_{12}a_{21}\hat{a}_{11}\\
& - h^3H_{pq}^{2}H_{qp}^{1}H_{qp}^{2}a_{11}a_{22}\hat{a}_{22} + h^3H_{pq}^{2}H_{qp}^{1}H_{qp}^{2}a_{12}a_{21}\hat{a}_{22} + h^2H_{pp}^{1}H_{qq}^{1}a_{11}\hat{a}_{11} + h^2H_{pp}^{1}H_{qq}^{2}a_{12}\hat{a}_{21}  \\
&+ h^2H_{pp}^{2}H_{qq}^{2}a_{22}\hat{a}_{22} + h^2H_{pq}^{1}H_{pq}^{2}\hat{a}_{11}\hat{a}_{22} - h^2H_{pq}^{1}H_{pq}^{2}\hat{a}_{12}\hat{a}_{21} - h^2H_{pq}^{1}H_{qp}^{1}a_{11}\hat{a}_{11} - h^2H_{pq}^{1}H_{qp}^{2}a_{22}\hat{a}_{11} \\
&- h^2H_{pq}^{2}H_{qp}^{1}a_{11}\hat{a}_{22} - h^2H_{pq}^{2}H_{qp}^{2}a_{22}\hat{a}_{22} + h^2H_{qp}^{1}H_{qp}^{2}a_{11}a_{22} - h^2H_{qp}^{1}H_{qp}^{2}a_{12}a_{21} - hH_{pq}^{1}\hat{a}_{11} - hH_{pq}^{2}\hat{a}_{22} \\
& + hH_{qp}^{1}a_{11}+ hH_{qp}^{2}a_{22} + h^2H_{pp}^{2}H_{qq}^{1}a_{21}\hat{a}_{12}+ 1)
\end{aligned}
\end{equation*}
\end{document}